\documentclass[11pt,english,twoside]{amsart}
\usepackage[latin9]{inputenc}
\usepackage[a4paper]{geometry}
\usepackage{color}
\usepackage{babel}
\usepackage{amsmath}
\usepackage{amsthm}
\usepackage{amssymb}
\usepackage{esint}
\usepackage[unicode=true,pdfusetitle,
 bookmarks=true,bookmarksnumbered=false,bookmarksopen=false,
 breaklinks=false,pdfborder={0 0 1},backref=false,colorlinks=false]
 {hyperref}

\date{\today}

\makeatletter
\theoremstyle{plain}
\newtheorem{thm}{\protect\theoremname}
  \theoremstyle{definition}
  \newtheorem{defn}[thm]{\protect\definitionname}
  \theoremstyle{plain}
  \newtheorem{cor}[thm]{\protect\corollaryname}
  \theoremstyle{plain}
  \newtheorem{lem}[thm]{\protect\lemmaname}
  \theoremstyle{plain}
  \newtheorem{prop}[thm]{\protect\propositionname}

\usepackage{babel}
 \usepackage{amsrefs}
\usepackage{amsfonts}\usepackage{amscd}\usepackage{mathrsfs}
\usepackage{graphicx}\usepackage{color}\usepackage{dsfont}
\usepackage{xcolor}\usepackage{lipsum}
\usepackage{tikz}
\usetikzlibrary{arrows,chains,matrix,positioning,scopes}
\usepackage{graphics}
\usepackage{cite}
\usepackage{hyperref}
\usepackage{cases}
\usepackage{bbm}
\usepackage{graphicx}
\usepackage{float}
\usepackage{mathrsfs}
\usepackage{enumerate}
\usepackage{appendix}
\usepackage{times}
\usepackage{microtype}
\usepackage{mathrsfs}
\usepackage{tikz}

\usepackage{mdef}
  \providecommand{\corollaryname}{Corollary}
  \providecommand{\definitionname}{Definition}
  \providecommand{\lemmaname}{Lemma}
  
\providecommand{\theoremname}{Theorem}

\makeatother

  \providecommand{\corollaryname}{Corollary}
  \providecommand{\definitionname}{Definition}
  \providecommand{\lemmaname}{Lemma}
  \providecommand{\propositionname}{Proposition}
\providecommand{\theoremname}{Theorem}

\begin{document}

\title{Path-by-path regularization by noise for scalar conservation laws}

\author{Khalil Chouk, Benjamin Gess}

\begin{abstract}
We prove a path-by-path regularization by noise result for scalar conservation laws. In particular, this proves regularizing properties for scalar conservation laws driven by fractional Brownian motion and generalizes the respective results obtained in \cite{GS14-2}. We introduce a new path-by-path scaling property which is shown to be sufficient to imply regularizing effects.  
\end{abstract}

\maketitle

\section{Introduction}

We prove regularity estimates for solutions to scalar conservation laws with rough flux of the type 
\begin{equation}\label{eq:intro_burgers}\begin{aligned}
\partial_{t}u+\sum_{i=1}^{d}\partial_{x_{i}}A^{i}(u)\circ\frac{\dd w_{t}^{i}}{\dd t}&=0\quad\text{on }\T^{d}\\
u(0)&=u_0\in L^\infty(\T^d),
\end{aligned}\end{equation}
where $\T^d$ is the $d$-dimensional torus, $w=(w^{1},...,w^{d})\in C([0,T];\R^{d})$ is a continuous function satisfying an irregularity condition and $A=(A^{1},...,A^{d})\in C^{2}(\mathbb{R},\mathbb{R}^{d})$ is supposed to satisfy a non-degeneracy condition detailed below. For the sake of simplicity, in the introduction we restrict to the model case 
\begin{equation}\label{eq:intro_burgers_3}
\partial_{t}u+ \frac{1}{2} \partial_{x}u^2\circ\frac{\dd w_{t}}{\dd t}=0\quad\text{on }\T.
\end{equation}

Regularizing effects of noise for scalar conservation laws of the type \eqref{eq:intro_burgers} 
%\begin{equation}\label{eq:GS}
%\partial_{t}u+\sum_{i=1}^{d}a^{i}(u)\partial_{x_{i}}u\circ\dd  \b_{t}^{i}=0\quad\text{on }\T^{d},
%\end{equation} 
have been first observed in \cite{GS14-2}, where the case of Brownian motion, that is, $w = \b = (\b^1,\dots,\b^d)$ being a standard Brownian motion has been considered. In this work it has been shown that bounded quasi-solutions $u$ to \eqref{eq:intro_burgers_3} satisfy $u \in L^1([0,T];W^{\l,1}(\T))$ for every $\l < \frac{1}{2}$, $\P$-a.s.. In contrast, in the deterministic case, that is $w(t)=t$ in \eqref{eq:intro_burgers_3}, it has been shown in \cite{DLW03} that there exist bounded quasi-solutions\footnote{We also refer to \cite{JP02} for the construction of entropy solutions to \eqref{eq:intro_burgers_3} with $w(t)=t$ and an $L^1$-forcing with limited regularity.} with $u \not\in L^1([0,T];W^{\l,1}(\T))$ for every $\l>\frac{1}{3}$. The analysis of \cite{GS14-2,LPS13-2} relies on probabilistic arguments making use of the fact that $\b$ has independent increments and of the scaling properties of Brownian motion. It is therefore not clear from \cite{GS14-2} if the regularizing effect of Brownian motion can be characterized in terms of its path properties. A partial answer has been given by the results of \cite{GG16} which imply, as a special case, that 
%where path-by-path regularity estimates for SPDE of the general type 
%  $$du + \frac{\sigma}{2}\partial_x u^2\circ \frac{dw_t}{dt} = \frac{1}{12}\partial_{xx} u^3 $$
%have been derived. In particular, it follows from \cite{GG16} that 
the \textit{entropy} solution to \eqref{eq:intro_burgers_3} satisfies a path-by-path estimate of the form
  $$\|u(t)\|_{W^{1,\infty}(\T)} \le \big(\max_{0\le s\le t} (w(s) - w(t)) \wedge (w(t) - \min_{0\le s\le t} w(s))\big)^{-1} ,$$
where $w$ is a continuous function. In particular, in the case of $w=\b$ being a Brownian motion this implies that $u(t)$ is Lipschitz continuous for all times $t>0$, $\P$-a.s.. An important subtlety in this result is that the $\P$-zero set depends on the time $t>0$. In fact, for $\P$-almost every fixed realization of the solution $u(\cdot,\omega)$ there will be times $t>0$ where shocks appear and thus $x\mapsto u(t,x,\omega)$ is not Lipschitz continuous. However, the type of regularizing effects used in \cite{GS14-2} and in \cite{GG16} are of different nature. While \cite{GS14-2} relies on averaging techniques and thus on an increased speed of averaging due to Brownian scaling, the effect in \cite{GG16} relies on (strict) convexity of the flux function and dependence of the direction of the flux on $w$. 
 
This leads to the two main questions addressed in this work: First, to classify the properties of the paths of the Brownian motion leading to the regularizing effect observed in \cite{GS14-2} and to thus obtain a better understanding of the interplay of the deterministic and stochastic averaging in this case. Second, is it possible to establish path-by-path versions of the results of  \cite{GS14-2} for a uniformly chosen $\P$-zero set, in particular, independent of the time $t>0$ (in contrast to \cite{GG16})? The purpose of the present paper is to positively answer both of these tasks.

In order to characterize sufficient properties for continuous paths implying regularizing effects in \eqref{eq:intro_burgers}, in the first part of this paper we will rely on the notion of $(\rho,\gamma)$-irregularity introduced in \cite{CG16}: A path $w\in C([0,T];\R^d)$ is said to be $(\rho,\gamma)$-irregular if 
\begin{equation}\label{eq:intro_rhogamma}
  \sup_{a\in\R^{d}}\sup_{0\le s<t\le T}(1+|a|)^{\rho}\frac{|\int_{s}^{t}e^{i\langle a,w_{r}\rangle}\dd r|}{|t-s|^{\g}}<\infty.
\end{equation}
Our results (cf.\ Theorem \ref{thm:main} below) applied to the special case of \eqref{eq:intro_burgers_3} with $w$ being an $\eta$-H\"older continuous, $(\rho,\gamma)$-irregular path and $u$ being a bounded quasi-solution yield the following path-by-path regularization by noise result 
\begin{equation}\label{eq:intro_est_0}
u\in L^{1}([0,T];W^{\l,1}(\T)),
\end{equation}
for all 
$$\l<\frac{\rho(\eta+1)-(1-\g)}{(\rho\vee 1)(\eta+1)+(1-\g)}\wedge \frac{\rho+2(\rho\vee 1)}{(\rho\vee 1)(2\eta+1)+(1-\g)}.$$
For example, the above result may be applied for $w$ given by $w=g+\b^{H}$, for any function $g\in C^{1}([0,T])$ and fractional Brownian motion $\b^{H}$ with Hurst parameter $H\in(0,1)$. Note that the arguments of \cite{GS14-2} could not handle the presence of a deterministic perturbation $g$ of $w=\b$. For $H\in (0,\frac{1}{2}]$ this yields that essentially bounded quasi-solutions to 
\begin{equation}
\partial_{t}u+\frac{1}{2}\partial_{x}u^2\circ\frac{\dd \b^H_{t}}{\dd t}+\frac{\dot{g}_{t}}{2}\partial_{x}u^{2}=0\quad\text{on }\T,\label{eq:intro_burgers_pert}
\end{equation}
$\P$-a.s.~have regularity of the type 
\begin{equation}\label{eq:intro_est}
u\in L^{1}([0,T];W^{\l,1}(\T)),\quad\forall \l<\frac{1}{1+2H}.
\end{equation}
We first note that in the case of Brownian motion, that is $H=\frac{1}{2}$, we fully recover the probabilistic estimate given in \cite{GS14-2}, i.e.\ $u\in L^{1}([0,T];W^{\l,1}(\T))$ for all $\l<\frac{1}{2}$. Hence, the path-by-path estimates developed in this paper are as good as the probabilistic estimates from \cite{GS14-2}, which is somewhat surprising in view of the fact that the path-by-path theory for SDE developed in \cite{CG16} cannot reproduce the probabilistic results.  In particular, \eqref{eq:intro_est} shows that for $H\in(0,\frac{1}{2}]$ bounded quasi-solutions to \eqref{eq:intro_burgers_pert} are more regular than quasi-solutions in the deterministic case. The $\P$-zero set in \eqref{eq:intro_est} is universal\footnote{Note that via \eqref{eq:intro_rhogamma} the zero set may depend on $g$.}, that is, it does not depend on the initial datum of the problem \eqref{eq:intro_burgers_pert}, nor on the flux $A$ and it is characterized in terms of the path property \eqref{eq:intro_rhogamma}.

We further note that $\l<\frac{1}{1+2H}\uparrow 1$ for $H\downarrow0$. Hence, a higher irregularity of the driving path induces a stronger regularizing effect. In particular, this allows us to analyze the interplay in between the degeneracy behavior of the flux $A$ in \eqref{eq:intro_burgers} and the irregularity of the path $w$ (cf.\ Remark \ref{rmk:interplay} below). In the case of entropy solutions, \eqref{eq:intro_est} implies 
\[
u(t)\in W^{\l,1}(\T),\quad\forall t>0,\ \l<\frac{1}{1+2H}.
\]
Hence, in the limit $H\downarrow0$ we recover the optimal regularity estimate $u(t)\in W^{1-\ve,1}(\T)$ for $\ve>0$. Indeed, note that for each fixed $H\in(0,1)$ shocks still appear an thus $u(t) \not\in W^{1,1}(\T)$ for some $t > 0$. 

The proof of the first main result (Theorem \ref{thm:main}) makes use of both the $\rho$ and $\gamma$ indices in the definition of  $(\rho,\gamma)$-irregularity. In contrast, an inspection of the proof of \cite{GS14-2} unveils that apart from independence of increments only the scaling property of Brownian motion is used. This suggests that the condition \eqref{eq:intro_rhogamma} may not be optimal in this setting. In addition, the condition \eqref{eq:intro_rhogamma} is not easy to check in examples (cf.\ \cite{CG16}). This motivates the second part of this work. In this part we develop a second proof, quite different from \cite{GS14-2}, of regularity of solutions to \eqref{eq:intro_burgers_3}. This second approach avoids the use of Fourier transformations, relying on real analysis only. As a consequence, this allows to replace the condition \eqref{eq:intro_rhogamma} by a new path-by-path condition for $w$ which is close to a pure scaling condition: Assume that there is a $\iota\in[\frac{1}{2},1]$ such that for every $\a\in(-1,0)$ and $\l\ge1$ we have
\begin{equation}
\int_0^Tdr\,\int_{0}^{T-r}dt\,e^{-\l t}|w_{t}^{r}|^{\a}\lesssim \l^{-1-\iota\a}.\label{eq:intro_real_noise_cdt}
\end{equation}
It is not difficult to see that $(\g,\rho)$-irregularity implies \eqref{eq:intro_real_noise_cdt}. In addition, the verification of \eqref{eq:intro_real_noise_cdt} with $\iota>H$ for fractional Brownian motion requires only a few lines of proof.
 Under assumption \eqref{eq:intro_real_noise_cdt} for $w$ being $\eta$-H\"older continuous we prove that bounded quasi-solutions to \eqref{eq:intro_burgers_3} satisfy
\begin{equation}\label{eq:intro_est_2}
u\in L^{1}_{loc}((0,T);W^{\l,1}(\T)),\quad\forall \l<\frac{1+\eta-\iota}{1+\eta+\iota },
\end{equation}
which, in the case of fractional Brownian motion $w=\b^H$ recovers \eqref{eq:intro_est}.

\subsubsection*{Structure of the paper}
The first main result, Theorem \ref{thm:main}, relying on $(\rho,\gamma)$-irregularity is stated and proved in Section \ref{sec:main}. In Section \ref{sec:real} the path-by-path scaling condition \eqref{eq:intro_real_noise_cdt} is introduced, verified for fractional Brownian motion and an alternative proof of regularity of solutions to \eqref{eq:intro_burgers_3} is given in Theorem \ref{thm:2nd}.

\subsection*{Notation}

For a function $w:[0,T]\to\R^d$ we let $w_{t}:=w(t)$ and $w_{t}^{s}=w_{s+t}-w_{s}$. Given two vectors $a,b\in\R^d$ we define $(a_. b)_i = a_ib_i$ to be the componentwise product and $\langle a,b \rangle$ to be the inner product. We will often use the shorthand notation $\int_v\ \psi(v) = \int_\R dv\ \psi(v) $. We let $\mcM([0,T]\times\R^d\times\R)$ be the space of all locally finite Radon measures on $[0,T]\times\R^d\times\R$ and $\mcM_{TV}([0,T]\times\R^d\times\R)$ be the subspace of measures $m$ with finite total mass $\|m\|_{TV}$. Further, let $C^\eta([0,T];\R^d)$ be the space of $\eta$-H\"older continuous functions with norm $\|\cdot\|_\eta$ and $BV_v = BV(\R_v)$ be the space of functions in $L^1_v$ of bounded variation. We will often use the convention $L^p_{t,x} := L^p([0,T]\times \R^d)$, $L^p_{t,x}L^q_v := L^p([0,T]\times \R^d_x;L^q(\R_v))$ and analogously for Sobolev spaces and spaces of measures. For $s \in \R$, $p\ge 1$ we let $\dot{W}_{x}^{\s,p}$, $W_{x}^{\s,p}$ denote the homogeneous and inhomogeneous Sobolev spaces respectively and we set $H^{s} := W^{s,2}$. For a distribution $f$ on $\R^d$ we let $\hat{f}:=\mcF f $ be its Fourier transform. The Fourier transform is taken with respect to the $x$-variable unless specified otherwise. We follow the usual notational conventions concerning real interpolation along the lines of \cite{BL76}.

\section{Main result}\label{sec:main}

As suggested in \cite{LPS13,LPS14}, entropy solutions to \eqref{eq:intro_burgers} are defined by passing to the kinetic form. That is, setting 
\[
\chi(u,v)=\begin{cases}
1 & \text{if }0<v<u\\
-1 & \text{if }u<v<0\\
0 & \text{otherwise }
\end{cases}
\]
and $\chi(t,x,v):=\chi(u(t,x),v)$ we informally obtain that $\chi$ satisfies 
\begin{align}
\partial_{t}\chi(t,x,v)+\sum_{i=1}^{d}a^{i}(v)\partial_{x_{i}}\chi(t,x,v)\circ\frac{\dd w_{t}^{i}}{\dd t} & =\partial_{v}m(x,t,v),\label{eq:intro_burgers_2}\\
\chi(0,x,v) & =\chi^{0}(x,v).\nonumber 
\end{align}
In order to pass to a robust form, that is, to a form that makes sense for all continuous functions $w$, we test by testfunctions transported along the characteristics. That is, for any given test-function $\psi$ we have 
\begin{equation}\begin{aligned}\label{eq:def}
\int_{x,v}\,\chi(t,x,v)\psi(x-a(v){}_{.}w_{t},v) & =\int_{x,v}\,\chi(0,x,v)\psi(x,v)\\&\quad\quad-\int_{x,v,r}\,m(r,x,v)\partial_{v}(\psi(x-a(v){}_{.}w_{r},v)),\\
\chi(0,x,v) & =\chi^{0}(x,v).
\end{aligned}\end{equation}
This leads to 
\begin{defn}
   A map $u \in L^\infty([0,T]\times\T^d)\cap C([0,T];L^1(\T^d))$ is said to be a quasi-solution to \eqref{eq:intro_burgers} if there is a finite Radon measure $m$ on $[0,T]\times\T^d\times\R$ such that \eqref{eq:def} is satisfied for all $\psi\in C^\infty_c(\T^d,\times\R)$. If $m$ is a non-negative measure, then $u$ is called an entropy solution.
\end{defn}
We refer to \cite{LPS13} for the well-posedness of entropy solutions to  \eqref{eq:intro_burgers} and note that the proof given there for the Cauchy problem can be applied to the torus without essential change.

We aim to estimate the regularity of the solution $u$ based on averaging techniques. Concerning the general setup of the proof we follow \cite{GS14-2} which in turn adapted arguments of \cite{DV13}. However, in \cite{GS14-2} probabilistic arguments were used leading to an estimate in expectation only. In order to avoid such a probabilistic argument different estimates have to be found that allow to unveil the relation to $(\rho,\g)$-irregularity of paths. Informally, \eqref{eq:intro_burgers_2} is equivalent to 
\[
\partial_{t}\chi(t,x,v)+\langle a(v){}_{.}\nabla_{x}\chi(t,x,v),\circ\frac{\dd w_{t}}{\dd t}\rangle+\Delta^{\a}\chi(t,x,v)=\Delta^{\a}\chi(t,x,v)+\partial_{v}m(t,x,v).
\]
Passing to Fourier modes in $x$ we have 
\[
\partial_{t}\hat{\chi}(t,n,v)+i\langle n{}_{.}a(v)\hat{\chi}(t,n,v),\circ\frac{\dd w_{t}}{\dd t}\rangle+|n|^{2\a}\hat{\chi}(t,n,v)=|n|^{2\a}\hat{\chi}(t,n,v)+\partial_{v}\hat{m}(t,n,v).
\]
Hence, by a change of variable and setting $b_{n}=|n|^{2\alpha}$, 
\begin{align*}
\hat{\chi}(t,n,v)= & e^{i\langle a(v){}_{.}n,w_{t}\rangle-b_{n}t}\hat{\chi}(0,n,v)+\int_{0}^{t}ds\,e^{i\langle a(v){}_{.}n,w_{t}-w_{s}\rangle-b_{n}(t-s)}b_{n}\hat{\chi}(s,n,v)\\
 & +\int_{0}^{t}ds\,e^{i\langle a(v){}_{.}n,w_{t}-w_{s}\rangle-b_{n}(t-s)}\partial_{v}\hat{m}(s,n,v).
\end{align*}
Integrating in $v$ yields 
\begin{align}
\hat{u}(t,n)= & \int_{v}e^{i\langle a(v){}_{.}n,w_{t}\rangle-b_{n}t}\hat{\chi}(0,n,v)+\int_{v}\int_{0}^{t}ds\,e^{i\langle a(v){}_{.}n,w_{t}-w_{s}\rangle-b_{n}(t-s)}b_{n}\hat{\chi}(s,n,v)\nonumber \\
 & +\int_{v}\int_{0}^{t}ds\,e^{i\langle a(v){}_{.}n,w_{t}-w_{s}\rangle-b_{n}(t-s)}\partial_{v}\hat{m}(s,n,v)\label{eq:decomp}\\
=: & \hat{u}{}^{0}(t,n)+\hat{u}{}^{1}(t,n)+\hat{u}{}^{2}(t,n).\nonumber 
\end{align}
Hence, in the sense of distributions,
\begin{equation}\label{eq:decomp_u}
  u=u^0+u^1+u^2,
\end{equation}
with $u^i$ defined via  \eqref{eq:decomp}.

In the proof, the regularity of each $u^{i}$ will be estimated separately. The above stated decomposition of $\hat{u}(t,n)$ suggests that the possible regularizing properties of driving paths could be related to the behaviour of the appearing oscillating integrals. Motivated by this observation and \cite{CG16} we introduce 
\begin{equation}
\Phi_{s,t}^{w}(a):=\int_{s}^{t}e^{i\langle a,w_{r}\rangle}\dd r\label{eq:def_Phi}
\end{equation}
and recall the following notion of irregularity of a path. 
\begin{defn}
\label{def:Irregularity}Let $\rho,\gamma>0$ and $w\in C([0,T];\R^{d})$. Then $w$ is $(\rho,\gamma)$-irregular if 
\[
\|\Phi^{w}\|_{W^{\rho,\gamma}([0,T])}=\|\Phi^{w}\|_{\rho,\gamma}:=\sup_{a\in\R^{d}}\sup_{0\le s<t\le T}(1+|a|)^{\rho}\frac{|\Phi_{s,t}^{w}(a)|}{|t-s|^{\g}}<\infty.
\]
We say that $w$ is $\rho$-irregular if there exists $\g>1/2$ such that $w$ is $(\rho,\g)$-irregular.
\end{defn}

We will impose that the function $a:=A'=(a^{1},a^{2},...,a^{d})\in C^{1}(\mathbb{R},\mathbb{R}^{d})$ satisfies a non-degeneracy condition of the form
\begin{equation}
\inf_{e=(e_{1},...,e_{d}))\in\mathbb{R}^{d}}\max_{i=1,...,d}|e_{i}(a^{i}(v_{2})-a^{i}(v_{1}))|\geq c|v_{2}-v_{1}|^{\nu}\label{eq:a_nondeg}
\end{equation}
for some fixed $\nu\geq1$, $c>0$ and all $v_1,v_2 \in \R$. As in \cite{GS14-2} this non-degeneracy condition is weaker than the standard non-degeneracy condition used in the deterministic case (cf.\ e.g.\ \cite{JP02}). Indeed, we may choose $A(u)=(\frac{1}{2} u^2,\dots,\frac{1}{2} u^2)$ to get \eqref{eq:a_nondeg} with $\nu=1$, whereas this choice of $A$ does not satisfy the usual deterministic non-degeneracy condition (see \cite{GS14-2} for more details). Our main result is: 
\begin{thm}
\label{thm:main}Let $w\in C^{\eta}([0,T],\mathbb{R}^{d})$ for some $\eta>0$ be $(\rho,\g)$-irregular and let $u$ be a quasi-solution to \eqref{eq:intro_burgers}. Assume that $a$ satisfies \eqref{eq:a_nondeg} for some $\nu\ge1$. Then, for all $T>0$ and all 
%\[
%\l<\frac{\rho(\eta+1)-(1-\g)}{(\nu\rho\vee1)(\eta+1)+(1-\g)},
%\]
$$\l  <\frac{\rho(\eta+1)-(1-\g)}{(\nu\rho\vee 1)(\eta+1)+(1-\g)}\wedge \frac{\rho+2(\nu\rho\vee 1)}{(\nu\rho\vee 1)(2\eta+1)+(1-\g)},$$
we have 
\[
\int_{0}^{T}dt\,\|u(t)\|_{W^{\l,1}}\le C(\|u^{0}\|_{L_{x}^{1}}+\|u\|_{L_{t,x}^{1}}+\|w\|_{\eta}\|a'(v)m\|_{TV}),
\]
for some constant $C=C(\|\Phi^{w}\|_{\rho,\gamma})$. If $u$ is an entropy solution, then, in addition, 
\[
\|u(t)\|_{W^{\l,1}}<\infty\quad\text{for all }t>0.
\]
\end{thm}
By \cite[Theorem 1.7]{CG16} for a fractional Brownian motion $\b^{H}$ with Hurst parameter $H\in(0,1)$ and a H\"older continuous path $g\in C^{\a}([0,T])$ for some $\a\in(0,1)$, we have that $\P$-a.s.~$t\mapsto\b_{t}^{H}(\omega)+g_{t}$ is $(\rho,\gamma)$-irregular for all $\rho<\frac{1}{2H}$ and some $\gamma>\frac{1}{2}$ and $\b_{\cdot}^{H}(\omega)\in C^{\eta}([0,T])$ for every $\eta\in(0,H)$. Hence, an application of Theorem \ref{thm:main} yields the following: 
\begin{cor}
\label{cor:fbm}Assume that $A$ satisfies \eqref{eq:a_nondeg} for some $\nu\ge1$. Let $\b^{H}$ be a fractional Brownian motion in $\R^{d}$ with Hurst parameter $H\in(0,1)$, $g\in C^{\a}([0,T])$ for some $\a\in[H,1)$ and $u$ be a quasi-solution to 
\begin{equation}
\partial_{t}u+\sum_{j=1}^{d}\partial_{x_{j}}A^{j}(u)\circ\frac{\dd\b_{t}^{H,j}}{\dd t}+\sum_{j=1}^{d}\partial_{x_{j}}A^{j}(u)\circ\frac{\dd g}{\dd t}=0\quad\text{on }\T^{d}.\label{eq:intro_burgers-1}
\end{equation}
Then, for all $\l<\frac{1}{(\nu\vee2H)(H+1)+H}$, 
\begin{equation}
u\in L_{t}^{1}W_{x}^{\l,1}.\label{eq:reg_fbm}
\end{equation}
In particular, for 
\begin{equation}
\partial_{t}u+\frac{1}{2}\partial_{x}u^{2}\circ\frac{\dd\b_{t}^{H}}{\dd t}+\frac{1}{2}\partial_{x}u^{2}=0\quad\text{on }\T\label{eq:intro_burgers-1-1}
\end{equation}
and $H\le\frac{1}{2}$ we get \eqref{eq:reg_fbm} for all $\l<\frac{1}{1+2H}$. 
\end{cor}

\begin{remark}\label{rmk:interplay} Corollary \ref{cor:fbm} allows to analyze the interplay of non-degeneracy of the flux and irregularity of the noise. Indeed, for two choices $H_i \le \frac{\nu_i}{2}$ we get the same order of regularity as long as 
\[
\nu_{2}(H_{2}+1)+H_{2}=\nu_{1}(H_{1}+1)+H_{1}.
\]
Hence, a higher order degeneracy $\nu$ of the flux may be compensated by a more irregular path.
\end{remark}

\begin{remark} In \cite{GS14-2} it was initially claimed a regularity of order $\l<\frac{4}{5}$ in \eqref{eq:reg_fbm}, later corrected to $\l<\frac{1}{2}$ in \cite{GS17} with the same proof.   
\end{remark}

\subsection{Proof of the main Theorem}

In the following we set, for $a\in\R^{d}$, $b\in\R$, $w\in C([0,T];\R^{d})$, $s,t\ge\text{0}$, 
\[
\Psi_{s,t}^{w}(a,b)=\int_{s}^{t}e^{i\langle a,w_{r}\rangle-2br}\dd r.
\]
We note that 
\begin{align}
\Psi_{s,t}^{w}(a,b) & =\int_{s}^{t}e^{i\,\langle a,w_{r}\rangle-2br}\dd r\nonumber %\\ & 
=\int_{0}^{t-s}e^{i\langle a,w_{r+s}-2b(r+s)\rangle}\dd r\nonumber \\
 & =e^{i\langle a,w_{s}\rangle-2bs}\int_{0}^{t-s}e^{i\langle a,w_{r+s}-w_{s}\rangle-2br}\dd r\label{eq:phi-shift}\\
 & =e^{i\langle a,w_{s}\rangle-2bs}\Psi_{0,t-s}^{w^{s}}(a,b).\nonumber 
\end{align}
We note that 
\begin{equation}
\|\Phi^{w^{\tau}}\|_{W^{\rho,\gamma}([S;T])}=\|\Phi^{w}\|_{W^{\rho,\gamma}([S+\tau;T+\tau])}.\label{eq:shifting_irregularity}
\end{equation}
The regularity estimate will rely on an estimate of

\[
K^{w}(a,b):=\sup_{s\in[0,T]}|\Psi_{0,T-s}^{w^{s}}(a,b)|.
\]

\begin{lem}
\label{lem:k-est}Assume that $w$ is $(\rho,\g)$ -irregular. Then, for any $\theta<\gamma+1$ and $|b|\geq1$, $a\in\R^{d}$, 
\begin{align*}
|\Psi_{0,T-s}^{w^{s}}(a,b)| & \lesssim\|\Phi^{w}\|_{W^{\rho,\gamma}([0,T])}\frac{|b|^{1-\t}}{(|a|^{\rho}+1)},
\end{align*}
and 
\[
|K^{w}(a,b)|\lesssim\|\Phi^{w}\|_{W^{\rho,\gamma}([0,T])}\frac{|b|^{1-\t}}{(|a|^{\rho}+1)}.
\]
\end{lem}
\begin{proof}
We observe that, by integration by parts, 
\[
\Psi_{0,T-s}^{w^{s}}(a,b)=\Phi_{0,T-s}^{w^{s}}(a)e^{-2b(T-s)}+2b\int_{0}^{T-s}\Phi_{0,t}^{w^{s}}(a)e^{-2bt}\dd t,
\]
where $\Phi$ is defined as in \eqref{eq:def_Phi}. Thus, using $(\rho,\g)$ irregularity of $w$ we obtain that 
\[
|\Phi_{0,T-s}^{w^{s}}|=|\Phi_{s,T}^{w}(a)|\leq\|\Phi^{w}\|_{W^{\rho,\gamma}([0,T])}\frac{|T-s|^{\gamma}}{1+|a|^{\rho}},\quad s\in[0,T]
\]
and 
\[
|\Phi_{0,t}^{w^{s}}|\lesssim\|\Phi^{w}\|_{W^{\rho,\gamma}([0,T])}\frac{t^{\gamma}}{1+|a|^{\rho}}.
\]
Thus, for $\theta<\gamma+1$, 
\begin{align*}
|\Psi_{0,T-s}^{w^{s}}(a,b)|\lesssim & |\Phi_{0,T-s}^{w^{s}}(a)|e^{-2b(T-s)}+2|b|\int_{0}^{T-s}|\Phi_{0,t}^{w^{s}}(a)|e^{-2bt}\dd t,\\
\lesssim & \|\Phi^{w}\|_{W^{\rho,\gamma}([0,T])}\frac{(T-s)^{\gamma}}{(|a|^{\rho}+1)}\frac{1}{|2b|^{\g}(T-s)^{\g}}\\
 & +2|b|\int_{0}^{T-s}\|\Phi^{w}\|_{W^{\rho,\gamma}([0,T])}\frac{t^{\gamma}}{(|a|^{\rho}+1)}\frac{1}{t^{\theta}|2b|^{\theta}}\dd t\\
\lesssim & \|\Phi^{w}\|_{W^{\rho,\gamma}([0,T])}\frac{1}{(|a|^{\rho}+1)}\left(|b|^{-\g}+|b|^{1-\t}\int_{0}^{T-s}t^{\gamma-\t}\dd t\right)\\
\lesssim & \|\Phi^{w}\|_{W^{\rho,\gamma}([0,T])}\frac{|b|^{1-\t}}{(|a|^{\rho}+1)}\left(1+\int_{0}^{T-s}t^{\gamma-\t}\dd t\right)\\
\lesssim & \|\Phi^{w}\|_{W^{\rho,\gamma}([0,T])}\frac{|b|^{1-\t}}{(|a|^{\rho}+1)},
\end{align*}
where we have used the inequality $e^{-x}\lesssim x^{-p}$ for any $p>0$. Taking the supremum in $s$ finishes the proof. 
\end{proof}
We next note that the notion of $(\rho,\gamma)$-irregularity introduced above offers some flexibility in the choice of the parameters. In fact, the following Lemma shows that one may enhance the H\"older time regularity of $\Phi^{w}(a)$ by loosing a certain amount on the decay in $|a|$. 

\begin{lemma}\label{lemma:interpol} Let $w$ be a $(\rho,\gamma)$-irregular path. Then, for every $\kappa\in(0,1)$, $w$ is a $(\rho\kappa,1-\kappa(1-\gamma))$-irregular path. Moreover, we have that 
\[
\|\Phi^{w}\|_{\rho\kappa,1-\kappa(1-\gamma)}\leq2^{1-\k}\|\Phi^{w}\|_{\rho,\gamma}^{\kappa}.
\]
\end{lemma} 
\begin{proof}
It suffices to interpolate the bound 
\[
|\Phi_{s,t}^{w}(a)|\leq\|\Phi^{w}\|_{\rho,\gamma}\frac{|t-s|^{\gamma}}{1+|a|^{\rho}}
\]
with the trivial bound $|\Phi_{s,t}(a)|\leq|t-s|$. We get
\begin{align*}
|\Phi_{s,t}^{w}(a)| & \le\|\Phi^{w}\|_{\rho,\gamma}^{\k}\frac{|t-s|^{\gamma\k+(1-\k)}}{(1+|a|^{\rho})^{\k}}\\
 & \le2^{1-\k}\|\Phi^{w}\|_{\rho,\gamma}^{\k}\frac{|t-s|^{\gamma\k+(1-\k)}}{1+|a|^{\rho\k}}.
\end{align*}

\end{proof}
\begin{lemma}\label{lemma:avr-int} Let $f_{1},f_{2}:\R\to\R_{+}$ be measurable and $a\in C^{1}(\R,\mathbb{R}^{d})$ satisfying 
\begin{equation}
\inf_{e\in\mathbb{R}^{d},|e|=1}|e{}_{.}(a(v_{1})-a(v_{2}))|\geq c|v_{1}-v_{2}|^{\nu}\label{eq:ass_flux}
\end{equation}
for some $\nu\geq1$ and $c>0$. Then for every compact set $K\subseteq\R$, $n\in\Z^{d}\setminus\{0\}$ and every $\rho\in(0,\frac{1}{\nu})$ we have 
\[
\iint_{K\times K}\dd v_{1}\dd v_{2}\,f_{1}(v_{1})f_{2}(v_{2})\frac{1}{1+|n{}_{.}(a(v_{1})-a(v_{2}))|^{\rho}}\lesssim\|f_{1}\|_{L^{2}}\|f_{2}\|_{L^{2}}|n|^{-\rho}.
\]
%where 
%\begin{equation}
%c_{|n|}(\rho)\sim^{|n|\to+\infty}\left\{ \begin{aligned}\frac{1}{|n|^{\frac{1}{\nu}}}, &\quad\text{ for }\rho\nu>1\\
%\frac{\log|n|}{|n|^{\frac{1}{\nu}}}, &\quad \text{ for }\rho\nu=1\\
%, &\quad \text{ for }\rho\nu<1.
%\end{aligned}
%\right.
%\end{equation}
\end{lemma} 
\begin{proof}
%Using implies 
%\begin{align*}
%&\iint_{K\times K}\dd v_{1}\dd v_{2}\,f_{1}(v_{1})f_{2}(v_{2})\frac{1}{1+|n{}_{.}(a(v_{1})-a(v_{2}))|^{\rho}}\\
%&\qquad\le \iint_{K\times K}\dd v_{1}\dd v_{2}\,f_{1}(v_{1})f_{2}(v_{2})\frac{1}{1+c|n|^{\rho}|v_{1}-v_{2}|^{\nu\rho}}.
%\end{align*}
%Hence, 
Using first the Cauchy Schwartz inequality in $v_{1}$, then \eqref{eq:ass_flux}, then Young's inequality yields, 
\begin{align*}
&\iint_{K\times K}\dd v_{1}\dd v_{2}\,f_{1}(v_{1})f_{2}(v_{2})\frac{1}{1+|n{}_{.}(a(v_{1})-a(v_{2}))|^{\rho}} \\
&\le \|f_{1}\|_{L^{2}}\left\Vert \int_{K}dv_{2}\,f(v_{2})\frac{1}{1+|n{}_{.}(a(\cdot)-a(v_{2}))|^{\rho}} \right\Vert _{L^{2}} \\
&\le \|f_{1}\|_{L^{2}}\left\Vert \int_{K}dv_{2}\,f(v_{2})\frac{1}{1+c|n|^{\rho}|\cdot-v_{2}|^{\nu\rho}}\right\Vert _{L^{2}}
  \le \|f_{1}\|_{L^{2}}\|f_{2}\|_{L^{2}}\Big\|\frac{1}{1+c|n|^{\rho}|\cdot|^{\nu\rho}}\Big\|_{L^{1}}.
\end{align*}
The conclusion of the proof now follows from the observation 
\[
\int_{K}\dd v\frac{1}{1+c|n|^{\rho}|v|^{\nu\rho}}\lesssim\frac{1}{|n|^{\frac{1}{\nu}}}\int_{|v|\lesssim|n|^{\frac{1}{\nu}}}\dd v\,\frac{1}{1+|v|^{\nu\rho}}\lesssim |n|^{-\rho}.
\]
\end{proof}
We now proceed by estimating each term appearing in \eqref{eq:decomp} separately.

In the following we will first consider the case $\nu\rho<1$. In the end of the proof we will see that the case $\nu\rho\ge1$ can be reduced to this case.

\subsubsection{Estimate for the initial condition part}

In the following we estimate $\hat{u}{}^{0}(t,n)$. With $b_{n}=|n|^{2\alpha}$ as above, we have 
\[
\hat{u}^{0}(t,n)=\int_{v}e^{-i\langle a(v){}_{.}n,w_{t}\rangle-b_{n}t}\hat{\chi}^{0}(n,v).
\]
Taking the square and integrating in time implies 
\begin{align*}
\int_{0}^{T}dt\,|\hat{u}^{0}(t,n)|^{2} & =\int_{0}^{T}dt\,\left|\int_{v}e^{-i\langle a(v){}_{.}n,w_{t}\rangle-b_{n}t}\hat{\chi}^{0}(n,v)\right|^{2}\\
 & =\int_{0}^{T}dt\,\int_{v_{1}}\int_{v_{2}}e^{i\langle(a(v_{1})-a(v_{2})){}_{.}n,w_{t}\rangle-2b_{n}t}\overline{\hat{\chi}^{0}}(n,v_{1})\hat{\chi}^{0}(n,v_{2})\\
 & =\int_{v_{1}}\int_{v_{2}}\Psi_{0,T}^{w}((a(v_{1})-a(v_{2})){}_{.}n,b_{n})\overline{\hat{\chi}^{0}}(n,v_{1})\hat{\chi}^{0}(n,v_{2}).
\end{align*}
By Lemma \ref{lem:k-est} we have, for $0<\theta<\gamma+1$,
\begin{equation}
|\Psi_{0,T}^{w}((a(v_{1})-a(v_{2}))_{.}n,b_{n})|\lesssim\|\Phi^{w}\|_{\rho,\gamma}\frac{|b_{n}|^{1-\t}}{|(a(v_{1})-a(v_{2})){}_{.}n|^{\rho}+1}.\label{eq:phi-est-1}
\end{equation}
We obtain 
\[
\int_{0}^{T}dt\,|\hat{u}^{0}(t,n)|^{2}\lesssim\|\Phi^{w}\|_{\rho,\gamma}\int_{v_{1},v_{2}}\frac{|b_{n}|^{1-\t}}{|(a(v_{1})-a(v_{2})){}_{.}n|^{\rho}+1}|\hat{\chi}^{0}|(n,v_{1})|\hat{\chi}^{0}|(n,v_{2}).
\]
Since $u_0 \in L^\infty(\T^d)$ we have that $\hat\chi^0(n,\cdot)$ is compactly supported. Hence, Lemma \ref{lemma:avr-int} gives 
\begin{equation}
\int_{0}^{T}dt\,|\hat{u}^{0}(t,n)|^{2}\lesssim\|\Phi^{w}\|_{\rho,\gamma}|n|^{-2\a(\t-1)-\rho}\|\hat\chi^{0}(n,\cdot)\|_{L_{v}^{2}}^{2},\label{eq:u0-est}
\end{equation}
This implies 
%\[
%\int_{0}^{T}\dd t\|u^{0}(t)\|_{H^{\frac{\rho+2\a(\t-1)}{2}}}^{2}\lesssim\|\Phi^{w}\|_{\rho,\gamma}\|\chi\|_{L_{v,x}^{2}}.
%\]
%Hence, multiplying by $|n|^{-2\tau}$ yields 
\begin{equation}
\int_{0}^{T}\dd t\,\|u^{0}(t)\|_{H^{\frac{\rho+2\a(\t-1)}{2}}}^{2}\lesssim\|\Phi^{w}\|_{\rho,\gamma}\|\chi^0\|_{L_{v,x}^{2}}^2.\label{eq:u0-est-1}
%\int_{0}^{T}\|u^{0}(t)\|_{H^{\frac{\rho+2\a(\t-1)}{2}+\tau}}^{2}\dd t\lesssim\|\Phi^{w}\|_{\rho,\gamma}\|\chi\|_{L_{t,v}^{2}(H^{\tau})}^{2}.
\end{equation}

\subsubsection{Estimate for the integral part}

Recall 
\[
u^{1}(t,n)=b_{n}\int_{0}^{t}\dd s\int_{v}e^{-i\langle a(v){}_{.}n,w_{t}-w_{s}\rangle-b_{n}(t-s)}\hat{\chi}(s,n,v).
\]
Taking the square of the $L^{2}$ norm in time and expanding the square into a double integral yields 
\begin{align*}
\mathcal{I} & =\int_{0}^{T}\dd t\,|u^{1}(t,n)|^{2}\\
 & =b_{n}^{2}\int_{0}^{T}\dd t\int_{0}^{t}\int_{0}^{t}\dd s_{1}\dd s_{2}\int_{v_{1}}\int_{v_{2}}e^{i\langle n{}_{.}(a(v_{2})-a(v_{1})),w_{t}\rangle-2b_{n}t}e^{i\langle n{}_{.}a(v_{2}),w_{s_{2}}\rangle+b_{n}s_{2}}\overline{\hat{\chi}}(s_{2},n,v_{2})\\
 & \qquad e^{-i\langle n{}_{.}a(v_{1}),w_{s_{1}}\rangle+b_{n}s_{1}}\hat{\chi}(s_{1},n,v_{1}).
\end{align*}
Now we observe that we can split the domain of integration into two regions $s_{1}>s_{2}$ and $s_{2}<s_{1}$ which give the same contribution. Thus, 
\begin{align*}
\mathcal{I} & =2b_{n}^{2}\int_{0}^{T}\dd t\int_{0}^{t}\dd s_{1}\int_{0}^{s_{1}}\dd s_{2}\int_{v_{1}}\int_{v_{2}}e^{i\langle n{}_{.}(a(v_{2})-a(v_{1})),w_{t}\rangle}e^{-2b_{n}t}e^{-i\langle n_{.}a(v_{2}),w_{s_{2}}\rangle+b_{n}s_{2}}\overline{\hat{\chi}}(s_{2},n,v_{2})\\
 & \qquad e^{i\langle n_{.}a(v_{1}),w_{s_{1}}\rangle+b_{n}s_{1}}\hat{\chi}(s_{1},n,v_{1}).
\end{align*}
Then, using Fubini's Theorem to commute the time integrals yields 
\begin{align*}
\mathcal{I} & =2b_{n}^{2}\int_{0}^{T}\dd s_{1}\int_{0}^{s_{1}}\dd s_{2}\int_{v_{1},v_{2}}\Psi_{s_{1},T}^{w}(n{}_{.}(a(v_{2})-a(v_{1})),b_{n})e^{-i\langle n{}_{.}a(v_{2}),w_{s_{2}}\rangle+b_{n}s_{2}}\overline{\hat{\chi}}(s_{2},n,v_{2})\\
 & \qquad e^{i\langle n{}_{.}a(v_{1}),w_{s_{1}}\rangle+b_{n}s_{1}}\hat{\chi}(s_{1},n,v_{1}),
\end{align*}
where $\Psi_{s,T}^{w}$ is defined as above and by \eqref{eq:phi-shift} we have 
\begin{align*}
\Psi_{s_{1},T}^{w} & (n{}_{.}(a(v_{2})-a(v_{1})),b_{n})=e^{-2b_{n}s_{1}+i\langle n{}_{.}(a(v_{2})-a(v_{1})),w_{s_{1}}\rangle}\Psi_{0,T-s_{1}}^{w^{s_{1}}}(n{}_{.}(a(v_{2})-a(v_{1})),b_{n}).
\end{align*}
We now focus on the time integral part for which we use Fubini's theorem to obtain 
\begin{align*}
 & \Bigg|\int_{0}^{T}\dd s_{1}\int_{0}^{s_{1}}\dd s_{2}\left[\Psi_{0,T-s_{1}}^{w^{s_{1}}}(n{}_{.}(a(v_{2})-a(v_{1})),b_{n})e^{-b_{n}(s_{1}-s_{2})}e^{i\langle n{}_{.}(a(v_{2})-a(v_{1})),w_{s_{1}}\rangle}e^{i\langle n{}_{.}a(v_{1}),w_{s_{1}}\rangle}\right]\\
 & \qquad e^{-i\langle n{}_{.}a(v_{2}),w_{s_{2}}\rangle}\overline{\hat{\chi}}(s_{2},n,v_{2})\hat{\chi}(s_{1},n,v_{1}))\Bigg|\\
 & \leq\int_{0}^{T}\dd s_{2}|\hat{\chi}(s_{2},n,v_{2})|\int_{s_{2}}^{T}\dd s_{1}|\Psi_{0,T-s_{1}}^{w^{s_{1}}}(n{}_{.}(a(v_{2})-a(v_{1})),b_{n})|e^{-b_{n}(s_{1}-s_{2})}|\hat{\chi}(s_{1},n,v_{1})|.
\end{align*}
Using the Cauchy-Schwartz inequality for the integral over $s_{2}$ we bound this quantity by 
\begin{align*}
&\|\hat{\chi}(\cdot,n,v_{2})\|_{L^{2}([0,T])}\\
&\times \left[\int_{0}^{T}\dd s_{2}\left(\int_{s_{2}}^{T}\dd s_{1}|\Psi_{0,T-s_{1}}^{w^{s_{1}}}(n{}_{.}(a(v_{2})-a(v_{1})),b_{n})|e^{-b_{n}(s_{1}-s_{2})}|\hat{\chi}(s_{1},n,v_{1})|\right)^{2}\right]^{\frac{1}{2}}.
\end{align*}
By Young's inequality for convolutions this can be bounded by 
\[
\frac{1}{b_{n}}\|\chi(\cdot,n,v_{2})\|_{L^{2}([0,T])}\left[\int_{0}^{T}\dd s_{1}|\hat{\chi}(s_{1},n,v_{1})|^{2}|\Psi_{0,T-s_{1}}^{w^{s_{1}}}(n{}_{.}(a(v_{2})-a(v_{1})),b_{n})|^{2}\right]^{\frac{1}{2}},
\]
which can be bounded by 
\[
\frac{1}{b_{n}}\|\chi(\cdot,n,v_{1})\|_{L^{2}([0,T])}\|\chi(\cdot,n,v_{2})\|_{L^{2}([0,T])}\sup_{s_{1}\in[0,T]}|\Psi_{0,T-s_{1}}^{w^{s_{1}}}(n{}_{.}(a(v_{2})-a(v_{1})),b_{n})|.
\]
Hence, we conclude that 
\[
\int_{0}^{T} dt\,|u^{1}(t,n)|^{2}\leq2b_{n}\int_{v_{1}}\int_{v_{2}}\|\chi(\cdot,n,v_{1})\|_{L^{2}([0,T])}\|\chi(\cdot,n,v_{2})\|_{L^{2}([0,T])}K^{w}(n{}_{.}(a(v_{2})-a(v_{1})),b_{n}),
\]
where 
\[
K^{w}(a,b)=\sup_{s\in[0,T]}|\Psi_{0,T-s}^{w^{s}}(a,b)|.
\]

By Lemma \ref{lem:k-est} we have 
\[
|K^{w}(a,b)|\lesssim\|\Phi^{w}\|_{\rho,\gamma}\frac{|b|^{1-\theta}}{(1+|a|^{\rho})}
\]
for all $\theta<\gamma+1$. Since $u\in L^\infty_{t,x}$ we have that $v\mapsto\|\chi(t,n,v)\|_{L^{2}([0,T])}$ is compactly supported in $v$ uniformly in $t\in [0,T]$. Hence, using Lemma~\ref{lemma:avr-int}, we conclude that 
\begin{equation}\begin{split}
  \int_{0}^{T}dt\,|u^{1}(t,n)|^{2} 
 & \lesssim\|\Phi^{w}\|_{\rho,\gamma}b_{n}\int_{v_{1}}\int_{v_{2}}\|\chi(\cdot,n,v_{1})\|_{L^{2}([0,T])}\|\chi(\cdot,n,v_{2})\|_{L^{2}([0,T])}\\
 &\qquad\qquad\qquad\qquad\times\frac{b_{n}{}^{1-\theta}}{(1+|n{}_{.}(a(v_{2})-a(v_{1}))|^{\rho})}\\
 & \lesssim\|\Phi^{w}\|_{\rho,\gamma}b_{n}^{2-\t}|n|^{-\rho}\|\chi(\cdot,n,\cdot)\|_{L_{t,v}^{2}}^{2}\label{eq:u1-est}\\
 & \lesssim\|\Phi^{w}\|_{\rho,\gamma}|n|^{2\a(2-\t)-\rho}\|\chi(\cdot,n,\cdot)\|_{L_{t,v}^{2}}^{2}.
\end{split}\end{equation}
Hence, multiplying by $|n|^{-2\tau}$ yields 
\begin{equation}
\int_{0}^{T}dt\,\|u^{1}(t)\|_{H^{\frac{\rho-2\a(2-\t)}{2}+\tau}}^{2}\lesssim\|\Phi^{w}\|_{\rho,\gamma}\|\chi\|_{L_{t,v}^{2}(H^{\tau}_x)}^{2}.\label{eq:u1-est-1-1}
\end{equation}

\subsubsection{Estimate for the kinetic measure}

We consider
\begin{align*}
\int_{0}^{T}dt\,\<(-\D)^{\frac{\l}{2}}\vp,u^{2}(t)\> & =-\int_{0}^{T}dt\int_{v}\int_x\int_{0}^{t}ds\,\partial_{v}S_{A(v)}^{*}(s,t)((-\D)^{\frac{\l}{2}}\vp)m(x,s,v)\\
 & =-\int_{0}^{T}dt\int_{v}\int_x\int_{0}^{t}ds\,(w_{t}-w_{s})a'(v)DS_{A(v)}^{*}(s,t)((-\D)^{\frac{\l}{2}}\vp)m(s,x,v),
\end{align*}
where $\vp$ is an arbitrary smooth function and $S_{A(v)}^{*}(s,t)$ is the dual semigroup to $S_{A(v)}(s,t)$ which is given by the Fourier multiplier $e^{i\langle a(v){}_{.}n,w_{t}-w_{s}\rangle-b_{n}(t-s)}$. As in \cite{GS14-2} we have that 
\[
\|DS_{A(v)}^{*}(s,t)(-\D)^{\frac{\l}{2}}\vp\|_{\infty}\lesssim(t-s)^{-\frac{\l+1}{2\a}}\|\vp\|_{\infty}.
\]
Hence, using that $w$ is $\eta$-H\"older continuous by assumption and Young's inequality for convolutions, 
\begin{align*}
\int_{0}^{T}dt\,\<(-\D)^{\frac{\l}{2}}\vp,u^{2}(t)\> & \lesssim\|\vp\|_{\infty}\int_{0}^{T}dt\int_{0}^{t}ds\,|w_{t}-w_{s}|(t-s)^{-\frac{\l+1}{2\a}}\int_{v}\int_{x}|a'(v)||m|(s,x,v)\\
 & \lesssim\|w\|_{\eta}\|\vp\|_{\infty}\int_{0}^{T}dt\int_{0}^{t}ds\,(t-s)^{\eta-\frac{\l+1}{2\a}}\int_{v}\int_{x}|a'(v)||m|(s,x,v)\\
 & \lesssim\|w\|_{\eta}\|\vp\|_{\infty}\int_{0}^{T}dt\,t{}^{\eta-\frac{\l+1}{2\a}}\,\int_{0}^{T}ds\int_{v}\int_{x}|a'(v)||m|(s,x,v).
\end{align*}
If $\eta-\frac{\l+1}{2\a}>-1$ or equivalently 
\[
2\a(\eta+1)-1>\l
\]
we conclude 
\begin{align*}
\int_{0}^{T}dt\,\<(-\D)^{\frac{\l}{2}}\vp,u^{2}(t)\> & \lesssim\|w\|_{\eta}\|\vp\|_{\infty}\int_{0}^{T}\int_{v}\int_{x}|a'(v)||m|(s,x,v)\\
 & =\|w\|_{\eta}\|\vp\|_{\infty}\|a'(v) m\|_{TV}.
\end{align*}
Hence, for $\l<2\a(\eta+1)-1\wedge1$, 
\begin{equation}
\int_{0}^{T}dt\,\|u^{2}(t)\|_{W^{\l,1}}\lesssim\|w\|_{\eta}\|a'(v) m\|_{TV}.\label{eq:u2-est}
\end{equation}

\subsubsection{Conclusion}

Let us consider first the case when $\rho\nu<1$. We may now conclude the proof analogously to \cite{GS14-2}. For the reader's convenience we include some details. Let $\tau \in [0,\frac{1}{2}]$ such that $\frac{\rho-2\a(2-\t)}{2}+\tau\le \frac{2\a(\t-1)+\rho}{2}$ which is equivalent to $\tau\le \a$. By \eqref{eq:decomp_u}, \eqref{eq:u0-est-1}, \eqref{eq:u1-est-1-1}, \eqref{eq:u2-est} and Sobolev embeddings, we have
\begin{equation}\label{eq:u_decomp}
  u=u^{0}+u^{1}+u^{2}\in L_{t}^{1}W_{x}^{\l,1}
\end{equation}
if 
\[
\frac{\rho-2\a(2-\t)}{2}+\tau-\frac{d}{2}\ge\l-d,\quad\frac{\rho-2\a(2-\t)}{2}+\tau\ge\l,\quad1\wedge2\a(\eta+1)-1>\l.
\]
This is satisfied if 
\begin{align*}
\min\Big(1,\frac{\rho-2\a(2-\t)}{2}+\tau,2\a(\eta+1)-1\Big)> & \l.
\end{align*}
Optimizing the left hand side for $\a$ yields the choice 
\[
\a=\frac{\rho+2+2\tau}{2(2\eta+4-\t)},
\]
which satisfies $\tau\le \a$ iff $\tau \le \frac{\rho+2}{2(2\eta+3-\t)}.$ Hence, for $\tau \in [0,\frac{1}{2}\wedge \frac{\rho+2}{2(2\eta+3-\t)}] $ we obtain \eqref{eq:u_decomp} with
\[
\l<\frac{\rho(\eta+1)-2+\t}{2\eta+4-\t}+\frac{2(\eta+1)}{2\eta+4-\t}\tau\wedge1.
\]
We now bootstrap: Start with $\tau=0$ to get $\l_{0}:=\frac{\rho(\eta+1)-2+\t}{2\eta+4-\t}$ we then set 
\[
\l_{n+1}=\frac{\rho(\eta+1)-2+\t}{2\eta+4-\t}+\frac{2(\eta+1)}{2\eta+4-\t}\frac{\l_{n}}{2}\wedge1.
\]
Then $\l_{n}\uparrow\l_{*}$ with 
\begin{equation}
\l_{*}=\frac{\rho(\eta+1)-2+\t}{\eta+3-\t}\wedge1,\label{eq:lambda_star}
\end{equation}
where the iteration has to be stopped if the side-condition $\l_n \le \frac{\rho+2}{2\eta+3-\t}$ is reached. In conclusion, we have, for all 
  $$\l<\l_{*}\wedge \frac{\rho+2}{2\eta+3-\t}$$
that
\begin{equation}\label{eq:est}
\int_{0}^{T}dt\,\|u(t)\|_{W^{\l,1}}\lesssim C(\|\Phi^{w}\|_{\rho,\gamma})(\|u^{0}\|_{L_{x}^{1}}+\|u\|_{L_{t,x}^{1}}+\|w\|_{\eta}\|a'(v)m\|_{TV}).
\end{equation}
If $u$ is an entropy solution, by $L^{1}$-contractivity with respect to the initial condition this implies 
\[
\|u(t)\|_{W^{\l,1}}<\infty
\]
for all $t>0$. Since $\theta<1+\gamma$ arbitrary we can set $\theta=1+\gamma$ in \eqref{eq:lambda_star} which yields \eqref{eq:est} for each
\[
\l<\frac{\rho(\eta+1)-(1-\g)}{(\eta+1)+(1-\g)}\wedge \frac{\rho+2}{2\eta+2-\g}.
\]
This finishes the proof of Theorem \ref{thm:main}.

\subsubsection{Reduction to the case \texorpdfstring{$\nu\rho<1$}{nu*rho<1} }

If $\nu\rho\ge1$ we choose $\k\in(0,\frac{1}{\nu\rho})$ and note that $w$ is $(\td\rho,\td\g)=(\k\rho,1-\k(1-\g))$-irregular by Lemma \ref{lemma:interpol}. Since then $\td\rho\nu=\k\rho\nu<1$ we can apply the estimates from the case $\nu\rho<1$. We obtain \eqref{eq:est} for all
\begin{align*}
\l & <\frac{\td\rho(\eta+1)-1+\td\g}{\eta+2-\td\g}\wedge \frac{\td\rho+2}{2\eta+1+(1-\td\g)}\\
 & =\frac{\k\rho(\eta+1)-\k(1-\g)}{\eta+1+\k(1-\g)}\wedge \frac{\k\rho+2}{2\eta+1+\k(1-\g)}.
\end{align*}
Choosing $\k\approx\frac{1}{\nu\rho}$ yields \eqref{eq:est} for all
\begin{align*}
\l & <\frac{\rho(\eta+1)-(1-\g)}{\nu\rho(\eta+1)+(1-\g)}\wedge \frac{\rho+2\nu\rho}{\nu\rho(2\eta+1)+(1-\g)}.
\end{align*}

\section{Real space method}\label{sec:real}

In the above sections we have shown that the notion of $(\rho,\gamma)$-irregularity of a path gives us a sufficient notion to capture path-by-path regularizing properties. In order to further analyze the role played by this notion of irregularity, in this section we introduce an alternative approach to averaging principles and thus path-by-path regularization by noise. The merit of this alternative approach is that it does not rely on Fourier methods and therefore it does not lead to oscillating random integrals as they appear in the definition of $(\rho,\gamma)$-irregularity. This leads us to an alternative (and apparently weaker) condition of irregularity of a path. This approach is motivated by \cite{JV04}. The disadvantage at the current stage is that a generalization to general flux and multiple dimension does not seem immediate.   

We again consider
\begin{equation}
\partial_{t}u+\frac{1}{2}\partial_{x}u^{2}\circ dw_{t}=0\label{eq:stoch_Burgers}
\end{equation}
 in its kinetic form, and restrict to one spatial dimension. Informally, we have
\[
\partial_{t}\chi(t,x,v)+v\partial_{x}\chi(t,x,v)\circ\frac{\dd w_{t}}{\dd t}=\partial_{v}m(t,x,v).
\]
For simplicity, we localize in time, that is, for $\vp \in C^\infty_c(0,T)$ we set $\td\chi := \vp\chi$, $\td m := \vp m$  which solves
\[
\partial_{t}\td\chi(t,x,v)+v\partial_{x}\td\chi(t,x,v)\circ\frac{\dd w_{t}}{\dd t}=\partial_{v}\td m(t,x,v)+\dot\vp(t) \chi(t,x,v).
\]
Rewrite as, for $\l>0$,
\[
\partial_{t}\td\chi(t,x,v)+v\partial_{x}\td\chi(t,x,v)\circ\frac{\dd w_{t}}{\dd t}+\l\td\chi(t,x,v)=\partial_{v}\td m(t,x,v)+\l\td\chi(t,x,v)+\dot\vp(t) \chi(t,x,v).
\]
Then,
\begin{equation}\begin{aligned}
\td\chi(t,x,v)= & \int_{0}^{t}ds\,e^{-\l(t-s)}(\partial_{v}\td m)(s,x-vw_{s,t},v)\label{eq:real_decomp} +\l\int_{0}^{t}ds\,e^{-\l(t-s)}\td\chi(s,x-vw_{s,t},v) \\
 &+ \int_{0}^{t}ds\,e^{-\l(t-s)}\dot\vp(s)\chi(s,x-vw_{s,t},v). %\nonumber 
\end{aligned}\end{equation}
We want to estimate the regularity of velocity averages
\[
u^{\phi}:=\int_{v}\chi\phi
\]
for $\phi\in C_{c}^{\infty}(\R)$. We thus aim to understand properties of the random $X$-ray transform operator
\[
Tg(x,t):=\int_{0}^{t}ds\int_{v}g(s,x-vw_{s,t},v)\phi(v)e^{-\l(t-s)},
\]
with $g\in\mcM([0,T]\times\R\times\R)+L^{\infty}([0,T]\times\R\times\R)$ , $\phi\in C_{c}^{\infty}(\R)$.

\subsection{A path-by-path scaling condition}

We will work with the following path-by-path scaling condition

\textbf{Assumption on the noise:} There is a $\iota\in[\frac{1}{2},1]$, $C\ge 0$ such that for every $\a\in(-1,0)$ and $\l\ge1$ we have
\begin{equation}
\int_0^Tdr\ \int_{0}^{T-r}dt\ e^{-\l t}|w_{t}^{r}|^{\a}\le C\l^{-1-\iota\a}.\label{eq:real_noise_cdt}
\end{equation}

The deterministic case ($w(t)=t$) corresponds to $\iota=1$. 
\begin{prop}
\label{prop:fbm-new}Let $w$ be a fractional Brownian motion with Hurst parameter $H\in(0,1)$. Then, for $\P$-a.e.~$\o\in\O$ the path $t\mapsto w_{t}(\omega)$ satisfies \eqref{eq:real_noise_cdt} for every $\iota>H$.
\end{prop}
\begin{proof}
Using that $e^{-x} \lesssim x^{-\t}$ for each $x\ge 0$, $\t\ge 0$ we observe that
$$
\int_{0}^{T}dr\ \int_{0}^{T-r}dt\ e^{-\l t}|w_{t}^{r}|^{\a}
\lesssim\l^{-\theta}\int_{0}^{T}dr\ \int_0^{T-r}dt\ t^{-\theta}|w^r_t|^\a.
$$
Hence, we need only to check that $K:=\int_{0}^{T}dr\ \int_0^{T-r}dt\ t^{-\theta}|w^r_t|^\a$ is finite $\P$-almost surely. Indeed,
$$
\mathbb E(K)\lesssim \int_{0}^{T}dr\ \int_0^{T-r}t^{-\theta+H\alpha}\dd t
$$
is finite under the condition that $\theta<1+H\a$.
\end{proof}
\begin{prop}
\label{prop:irregular-new}Assume that $w$ is $(\rho,\gamma)$-irregular with $\gamma>\frac{1}{2}$. Then $w$ satisfies \eqref{eq:real_noise_cdt} for $\a\in(-\rho\vee-1,0)$ with $\iota=\frac{1}{2\rho}$. %{[}TODO: How about other direction. Simple counterexample?{]}
\end{prop}
\begin{proof}
From \cite[Proposition 1.29]{BCD11} we recall that, for $\a\in(-1,0)$,
\begin{align*}
|x|^{\a} & =C_{\a}\mcF^{-1}(|\cdot|^{-\a-1})(x)\\
 & =\frac{C_{\a}}{2\pi}\int dy\,e^{ixy}|y|^{-\a-1}.
\end{align*}
Hence,
\begin{align*}
\int_{0}^{T-r}dt\,e^{-\l t}|w_{t}^{r}|^{\a} & =\frac{C_{\a}}{2\pi}\int_{0}^{T-r}dt\,e^{-\l t}\int dy\,e^{iw_{t}^{r}y}|y|^{-\a-1}\\
 & =\frac{C_{\a}}{2\pi}\int dy\,|y|^{-\a-1}\int_{0}^{T-r}dt\,e^{-\l t}e^{iw_{t}^{r}y}.
\end{align*}
Now
\begin{align*}
\int_{0}^{T-r}dt\,e^{-\l t}e^{iw_{t}^{r}y} 
& =e^{-\l \cdot}(\int_{0}^{\cdot}ds\,e^{iw_{s}^{r}y})|{}_{0}^{T-r}+\l\int_{0}^{T-r}dt\,e^{-\l t}(\int_{0}^{t}ds\,e^{iw_{s}^{r}y})\\
 & =e^{-\l(T-r)}(\int_{0}^{T-r}ds\,e^{iw_{s}^{r}y})+\l\int_{0}^{T-r}dt\,e^{-\l t}(\int_{0}^{t}ds\,e^{iw_{s}^{r}y})\\
 & \lesssim\frac{\|w\|_{\rho,\g}}{1+|y|^{\rho}}\left(e^{-\l(T-r)}|T-r|^{\gamma}+\l\int_{0}^{T-r}dt\,e^{-\l t}t^{\gamma}\right)\\
 & \lesssim\frac{\|w\|_{\rho,\g}}{1+|y|^{\rho}}\l^{-\gamma}.
\end{align*}
By Lemma \ref{lemma:interpol} for $\ve>0$ small enough, $w$ is $(\rho^{*},\gamma^{*}):=(-\a+\ve,1+\frac{\a}{2\rho}<1-\frac{-\a+\ve}{\rho}(1-\g))$-irregular. Using the above with these indices, since $\a-\ve-\a-1=-1-\ve<-1$ we have  
\begin{align*}
\int_{0}^{T-r}dt\,e^{-\l t}|w_{t}^{r}|^{\a} & \lesssim\|w\|_{\rho,\g}\l^{-(1+\frac{\a}{2\rho})}\int dy\,\frac{|y|^{-\a-1}}{1+|y|^{\rho^*}}\\
 & \lesssim\|w\|_{\rho,\g}\l^{-1-\frac{\a}{2\rho}}.
\end{align*}
\end{proof}
In the case of fractional Brownian motion, by Theorem \cite[Theorem 1.4]{CG16} we have that $\b^{H}$ is $(\rho,\gamma)$-irregular for any $\rho<\frac{1}{2H}$ and some $\gamma>\frac{1}{2}$. Proposition \ref{prop:irregular-new} then implies that fractional Brownian motion satisfies \eqref{eq:real_noise_cdt} with any $\iota=\frac{1}{2\rho}>H$. However, the proof of this fact given in Proposition \ref{prop:fbm-new} is much simpler and thus underlines the relevance of condition \eqref{eq:real_noise_cdt}.

\subsection{Main result}
\begin{lem}
\label{lem:1} Assume \eqref{eq:real_noise_cdt}. For $\s\in[0,1)$ we have 
\[
T:L_{t}^{\infty}(L_{x,loc}^{1}(BV_{v}))\to L_{t}^{1}(\dot{W}_{x,loc}^{\s,1})
\]
with
\[
\|T\|_{L_{t}^{\infty}(L_{x,loc}^{1}(BV_{v}))\to L_{t}^{1}(\dot{W}_{x,loc}^{\s,1})}\lesssim\l^{-1+\iota\s},
\]
for all $\l\ge1$.\end{lem}
\begin{proof}
Let $g\in L_{t}^{1}(L_{x,loc}^{1}(BV_{v}))$. Then, in the sense of distributions,
\[
\partial_{x}g(s,x-vw_{s,t},v)=-\frac{1}{w_{s,t}}(\partial_{v}(g(s,x-vw_{s,t},v))+\frac{1}{w_{s,t}}(\partial_{v}g)(s,x-vw_{s,t},v).
\]
Thus,
\begin{align*}
 & \|\int_{v}g(s,x-vw_{s,t},v)\phi(v)\|_{\dot{W}_{x,loc}^{1,1}}=\|\partial_{x}\int_{v}g(s,x-vw_{s,t},v)\phi(v)\|_{L_{x,loc}^{1}}\\
 & \le\frac{1}{|w_{s,t}|}\|\int_{v}(\partial_{v}(g(s,x-vw_{s,t},v))\phi(v)\|_{L_{x,loc}^{1}}+\frac{1}{|w_{s,t}|}\|\int_{v}(\partial_{v}g)(s,x-vw_{s,t},v)\phi(v)\|_{L_{x,loc}^{1}}\\
 & \le\frac{1}{|w_{s,t}|}\|\int_{v}g(s,x-vw_{s,t},v)\dot{\phi}(v)\|_{L_{x,loc}^{1}}+\frac{1}{|w_{s,t}|}\|\int_{v}(\partial_{v}g)(s,x-vw_{s,t},v)\phi(v)\|_{L_{x,loc}^{1}}.
\end{align*}
Minkowski's inequality yields
\begin{align*}
 & \|\int_{v}g(s,x-vw_{s,t},v)\phi(v)\|_{\dot{W}_{x,loc}^{1,1}}\\
 & \le\frac{1}{|w_{s,t}|}\int_{v}\|g(s,\cdot,v)\dot{\phi}(v)\|_{L_{x,loc}^{1}}+\frac{1}{|w_{s,t}|}\int_{v}\|(\partial_{v}g)(s,\cdot,v)\phi(v)\|_{L_{x,loc}^{1}}\\
 & \lesssim\frac{1}{|w_{s,t}|}\left(\|g(s,\cdot,\cdot)\|_{L_{x,loc}^{1}(L_{v,loc}^{1})}+\|g(s,\cdot,\cdot)\|_{L_{x,loc}^{1}(\dot{BV_{v}})}\right)\\
 & =\frac{1}{|w_{s,t}|}\|g(s,\cdot,\cdot)\|_{L_{x,loc}^{1}(BV_{v})}.
\end{align*}
We further have the bound
\begin{align*}
\|\int_{v}g(s,x-vw_{s,t},v)\phi(v)\|_{L_{x,loc}^{1}} & \le\int_{v}\|g(s,x-vw_{s,t},v)\phi(v)\|_{L_{x,loc}^{1}}\\
 & \le\|g(s,\cdot,\cdot)\|_{L_{x,loc}^{1}(L_{v,loc}^{1})}\\
 & \le\|g(s,\cdot,\cdot)\|_{L_{x,loc}^{1}(BV_{v})}.
\end{align*}
By complex interpolation these bounds yield, for $\s\in[0,1)$,
\[
\|\int_{v}g(s,x-vw_{s,t},v)\phi(v)\|_{\dot{W}_{x,loc}^{\s,1}}\le\frac{1}{|w_{s,t}|^{\s}}\|g(s)\|_{L_{x,loc}^{1}(BV_{v})}.
\]
Hence,
\begin{align*}
 & \|Tg(t)\|_{\dot{W}_{x,loc}^{\s,1}}\\
 & \le\left\Vert \int_{0}^{t}ds\int_{v}g(s,x-vw_{s,t},v)\phi(v)e^{-\l(t-s)}\right\Vert _{\dot{W}_{x,loc}^{\s,1}}\\
 & \le\int_{0}^{t}ds\|\int_{v}g(s,x-vw_{s,t},v)\phi(v)\|_{\dot{W}_{x,loc}^{\s,1}}e^{-\l(t-s)}\\
 & \le\int_{0}^{t}ds\frac{1}{|w_{s,t}|^{\s}}\|g(s)\|_{L_{x,loc}^{1}(BV_{v})}e^{-\l(t-s)}.
\end{align*}
We obtain (using the assumption on the noise)
\begin{align*}
\|Tg\|_{L_{t}^{1}\dot{W}_{x,loc}^{\s,1}} & \le\int_{0}^{T}dt\int_{0}^{t}ds\,\frac{1}{|w_{s,t}|^{\s}}\|g(s)\|_{L_{x,loc}^{1}(BV_{v})}e^{-\l(t-s)}\\
 & \le\int_{0}^{T}ds\,\left(\int_{0}^{T-s}dt\frac{e^{-\l t}}{|w_{0,t}^{s}|^{\s}}\right)\|g(s)\|_{L_{x,loc}^{1}(BV_{v})}\\
 & \le\l^{-1+\iota\s}\|g\|_{L_{t}^{\infty}L_{x,loc}^{1}(BV_{v})}.
\end{align*}
\end{proof}
\begin{lem}
\label{lem:2} Assume $w\in C^\eta([0,T])$ for some $\eta \in [0,1]$. For each $g\in\mcM_{t,x}\mcM_{v,loc}$ there are $h^{1},h^{2}\in\mcM([0,T]\times\R)$ such that
\[
T\partial_{v}g=\partial_{x}h^{1}+h^{2}
\]
with
\begin{align*}
\|h^{1}\|_{L_{t}^{1}\mcM_{x}} & \lesssim\l^{-1-\eta}\|g\|_{\mcM_{t,x}\mcM_{v,loc}}\\
\|h^{2}\|_{L_{t}^{1}\mcM_{x}} & \lesssim\l^{-1}\|g\|_{\mcM_{t,x}\mcM_{v,loc}},
\end{align*}
for all $\l\ge1$.\end{lem}
\begin{proof}
Let $g\in\mcM_{t,x}\mcM_{v,loc}$. Then
\begin{align*}
T(t)\partial_{v}g & =\int_{0}^{t}ds\,\int_{v}(\partial_{v}g)(s,x-vw_{s,t},v)\phi(v)e^{-\l(t-s)}\\
 & =\partial_{x}\int_{0}^{t}ds\,\int_{v}w_{s,t}(g(s,x-vw_{s,t},v))\phi(v)e^{-\l(t-s)}\\
 & +\int_{0}^{t}ds\,\int_{v}\partial_{v}(g(s,x-vw_{s,t},v))\phi(v)e^{-\l(t-s)}\\
 & =:\partial_{x}h^{1}(t)+h^{2}(t).
\end{align*}
Now, again using Minkowski's inequality, 
\begin{align*}
\|h^{1}(t)\|_{\mcM_{x}} & \le\int_{0}^{t}|w_{s,t}|\|g\phi\|_{\mcM_{x,v}}e^{-\l(t-s)}ds\\
\|h^{2}(t)\|_{\mcM_{x}} & \le\int_{0}^{t}\|g\dot{\phi}\|_{\mcM_{x,v}}e^{-\l(t-s)}ds.
\end{align*}
Hence, by assumption on the noise,
\begin{align*}
\|h^{1}\|_{L_{t}^{1}\mcM_{x}} & \le\int_{0}^{T}dt\,\int_{0}^{t}ds\,|w_{s,t}|\|g(s)\phi\|_{\mcM_{x,v}}e^{-\l(t-s)}\\
 & \le\int_{0}^{T}ds\,\left(\int_{0}^{T-s}dt\,t^\eta e^{-\l t}\right)\|g(s)\phi\|_{\mcM_{x,v}}\\
 & \le\l^{-1-\eta}\|g\|_{\mcM_{t,x}\mcM_{v,loc}}
\end{align*}
and 
\begin{align*}
\|h^{2}\|_{L_{t}^{1}\mcM_{x}} & \le\int_{0}^{T}dt\int_{0}^{t}ds\,\|g(s)\dot{\phi}\|_{\mcM_{x,v}}e^{-\l(t-s)}\\
 & \le\int_{0}^{T}ds\,\left(\int_{0}^{T-s}dt\,e^{-\l t}\right)\|g(s)\dot{\phi}\|_{\mcM_{x,v}}\\
 & \le\frac{1}{\l}\|g\|_{\mcM_{t,x}\mcM_{v,loc}}.
\end{align*}
\end{proof}
\begin{thm}\label{thm:2nd}
Let $u$ be a quasi-solution to \eqref{eq:stoch_Burgers}, $\phi\in C_{c}^{\infty}(\R)$ and suppose that $w\in C^\eta([0,T])$ satisfies assumption \eqref{eq:real_noise_cdt}. Then 
\[
u^{\phi}=\int_{v}\,\chi\phi\in L^{1}_{loc}((0,T);\dot{W}_{x,loc}^{s,1})
\]
for all $s<s_{*}=\frac{1+\eta-\iota}{1+\eta+\iota}$ with
\begin{align*}
\|u^{\phi}\|_{L_{loc}^{1}((0,T);\dot{W}_{x,loc}^{s,1})} & \lesssim\left(\|u\|_{L^\infty_tL^1_{x,loc}}+1\right)^{\frac{1+\eta}{\iota+1+\eta}}\left(\|m\|_{\mcM_{t,x}\mcM_{v,loc}}\vee1\right)^{\frac{\iota+\eta}{\iota+1+\eta}}+\|u^{\phi}\|_{L_{t}^{1}(\dot{W}_{x}^{-1,1})}.
\end{align*}
\end{thm}
\begin{proof}
From \eqref{eq:real_decomp} we have
\[
u^{\phi}=T\partial_vg+\l Tf_1 + Tf_2
\]
with $f_1=\td\chi$, $f_2=\dot\vp \chi$ , $g=\td m$. For simplicity we write $u$ instead of $u^{\phi}$ in the following. We note that 
\begin{align*}
f_1,f_2 & \in L_{t}^{\infty}L_{x,loc}^{1}BV_{v}\\
g=\td m & \in\mcM_{t,x}\mcM_{v,loc},
\end{align*}
with, for $i=1,2$,
\begin{align*}
 \|f_i\|_{L_{t}^{\infty}L_{x,loc}^{1}BV_{v}} \lesssim \|\chi\|_{L_{t}^{\infty}L_{x,loc}^{1}BV_{v}} \lesssim \|u\|_{L^\infty_tL^1_{x,loc}}+1 \\
 \|\td m\|_{\mcM_{t,x}\mcM_{v,loc}} \lesssim  \|m\|_{\mcM_{t,x}\mcM_{v,loc}}.
\end{align*}
 Following Lemma \ref{lem:2} we have
\[
Tg=h^{1}+h^{2}
\]
with, for arbitrary $\ve_{1},\ve_{2}>0$, 
\begin{align*}
\|h^{1}\|_{L_{t}^{1}(\dot{W}_{x}^{-1-\ve_{1},1})} & \lesssim \l^{-1-\eta}\| m\|_{\mcM_{t,x}\mcM_{v,loc}},\\
\|h^{2}\|_{L_{t}^{1}(W_{x}^{-\ve_{2},1})} & \lesssim \l^{-1}\| m\|_{\mcM_{t,x}\mcM_{v,loc}}.
\end{align*}
Since, for $0<\ve_{2}<<\ve_{1}$, 
\[
L_{t}^{1}(W_{x}^{-\ve_{2},1})\hookrightarrow(L_{t}^{1}(\dot{W}_{x}^{-1-\ve_{1},1}),L_{t}^{1}(\dot{W}_{x,loc}^{\s,1}))_{\frac{1}{1+\s},\infty}
\]
with 
\begin{align*}
K(h^{2},z):=\inf\{ & \|h^{2,1}\|_{L_{t}^{1}(\dot{W}_{x}^{-1-\ve_{1},1})}+z\|h^{2,2}\|_{L_{t}^{1}(\dot{W}_{x,loc}^{\s,1})}:\,h^{2}=h^{2,1}+h^{2,2},\\
 & \,h^{2,1}\in L_{t}^{1}(\dot{W}_{x}^{-1-\ve_{1},1}),h^{2,2}\in L_{t}^{1}(\dot{W}_{x,loc}^{\s,1})\}
\end{align*}
we have 
\begin{align*}
\sup_{z>0}z^{-\frac{1}{1+\s}}\,K(h^{2},z) & =:\|h^{2}\|_{(L_{t}^{1}(\dot{W}_{x}^{-1-\ve_{1},1}),L_{t}^{1}(\dot{W}_{x,loc}^{\s,1}))_{\frac{1}{1+\s},\infty}}\\
 & \lesssim\|h^{2}\|_{L_{t}^{1}(W_{x}^{-\ve_{2},1})}\lesssim\l^{-1}\|\td m\|_{\mcM_{t,x}\mcM_{v,loc}}.
\end{align*}
By Lemma \ref{lem:1} we have, $i=1,2$,
\[
\|Tf_i\|_{L_{t}^{1}(\dot{W}_{x,loc}^{\s,1})}\le\l^{-1+\iota\s}\|f_i\|_{L_{t}^{\infty}L_{x,loc}^{1}BV_{v}}.
\]
With  
\begin{align*}
K(u,z):=\inf\{ & \|u^{1}\|_{L_{t}^{1}(\dot{W}_{x}^{-1-\ve_{1},1})}+z\|u^{2}\|_{L_{t}^{1}(\dot{W}_{x,loc}^{\s,1})}:\,u=u^{1}+u^{2},\\
 & u^{1}\in L_{t}^{1}(\dot{W}_{x}^{-1-\ve_{1},1}),\,u^{2}\in L_{t}^{1}(\dot{W}_{x,loc}^{\s,1})\},
\end{align*}
since $u\in L_{t,x}^{1}\subseteq L_{t}^{1}(\dot{W}_{x}^{-1-\ve_{1},1})$ we have the trivial bound
\[
K(u,z)\le\|u\|_{L_{t}^{1}(\dot{W}_{x}^{-1-\ve_{1},1})}.
\]
It it therefore enough to restrict to $z\le1$ in the following estimates. For $h^{2}=h^{2,1}+h^{2,2}$ with $h^{2,1}\in L_{t}^{1}(\dot{W}_{x}^{-1-\ve_{1},1})$, $h^{2,2}\in L_{t}^{1}(\dot{W}_{x,loc}^{\s,1})$ we have
\begin{align*}
u & =h^{1}+h^{2}+\l Tf_1+ Tf_2\\
 & =h^{1}+h^{2,1}+h^{2,2}+\l Tf_1 + Tf_2
\end{align*}
and thus
\begin{align*}
K(u,z) & \le\|h^{1}+h^{2,1}\|_{L_{t}^{1}(\dot{W}_{x}^{-1-\ve_{1},1})}+z\|h^{2,2}+ \l Tf_1 + Tf_2\|_{L_{t}^{1}(\dot{W}_{x,loc}^{\s,1})}\\
 & \le\|h^{1}\|_{L_{t}^{1}(\dot{W}_{x}^{-1-\ve_{1},1})}+\|h^{2,1}\|_{L_{t}^{1}(\dot{W}_{x}^{-1-\ve_{1},1})}\\
 &+z\|h^{2,2}\|_{L_{t}^{1}(\dot{W}_{x,loc}^{\s,1})}+z\|\l Tf_1\|_{L_{t}^{1}(\dot{W}_{x,loc}^{\s,1})}+z\| Tf_2\|_{L_{t}^{1}(\dot{W}_{x,loc}^{\s,1})}.
\end{align*}
By definition of $K(h^{2},z)$ this yields
\begin{align*}
K(u,z) & \le\|h^{1}\|_{L_{t}^{1}(\dot{W}_{x}^{-1-\ve_{1},1})}+K(h^{2},z)+z\|\l Tf_1\|_{L_{t}^{1}(\dot{W}_{x,loc}^{\s,1})}+z\| Tf_2\|_{L_{t}^{1}(\dot{W}_{x,loc}^{\s,1})}\\
 & \lesssim\l^{-1-\eta}\| m\|_{\mcM_{t,x}\mcM_{v,loc}}+z^{\frac{1}{1+\s}}\l^{-1}\| m\|_{\mcM_{t,x}\mcM_{v,loc}}\\
 &+z\l^{\iota\s}\|\chi\|_{L_{t}^{\infty}L_{x,loc}^{1}BV_{v}}+z\l^{\iota\s-1}\|\chi\|_{L_{t}^{\infty}L_{x,loc}^{1}BV_{v}}.
\end{align*}
Equilibrating the first and third term yields (without loss of generality we may assume in the following that $\|\chi\|_{L_{t}^{\infty}L_{x,loc}^{1}BV_{v}}, \| m\|_{\mcM_{t,x}\mcM_{v,loc}} \ge 1$)
\[
\l^{-1-\eta}\|m\|_{\mcM_{t,x}\mcM_{v,loc}}=z\l^{\iota\s}\|\chi\|_{L_{t}^{\infty}L_{x,loc}^{1}BV_{v}}
\]
that is
\[
\l=z^{-\frac{1}{\iota\s+1+\eta}}\|m\|_{\mcM_{t,x}\mcM_{v,loc}}^{\frac{1}{\iota\s+1+\eta}}\|\chi\|_{L_{t}^{\infty}L_{x,loc}^{1}BV_{v}}^{-\frac{1}{\iota\s+1+\eta}}.
\]
Hence,
\begin{align*}
\l^{-1-\eta}&=z^{\frac{1+\eta}{\iota\s+1+\eta}}\|m\|_{\mcM_{t,x}\mcM_{v,loc}}^{-\frac{1+\eta}{\iota\s+1+\eta}}\|\chi\|_{L_{t}^{\infty}L_{x,loc}^{1}BV_{v}}^{\frac{1+\eta}{\iota\s+1+\eta}}\\
z^{\frac{1}{1+\s}}\l^{-1} & 
%=z^{\frac{1}{1+\s}+\frac{1}{\iota\s+1+\eta}}\|m\|_{\mcM_{t,x}\mcM_{v,loc}}^{-\frac{1}{\iota\s+1+\eta}}\|f\|_{L_{t}^{\infty}L_{x,loc}^{1}BV_{v}}^{\frac{1}{\iota\s+1+\eta}}\\
% &
=z^{\frac{1}{1+\s}+\frac{1}{\iota\s+1+\eta}}\|m\|_{\mcM_{t,x}\mcM_{v,loc}}^{-\frac{1}{\iota\s+1+\eta}}\|\chi\|_{L_{t}^{\infty}L_{x,loc}^{1}BV_{v}}^{\frac{1}{\iota\s+1+\eta}}\\
z\l^{\iota\s-1}&= z^\frac{2+\eta}{\iota\s+1+\eta} \|m\|_{\mcM_{t,x}\mcM_{v,loc}}^{\frac{\iota\s-1}{\iota\s+1+\eta}}\|\chi\|_{L_{t}^{\infty}L_{x,loc}^{1}BV_{v}}^{-\frac{\iota\s-1}{\iota\s+1+\eta}}.
\end{align*}
We obtain 

\begin{align*}
K(u,z) & \lesssim z^{\frac{1+\eta}{\iota\s+1+\eta}}\|m\|_{\mcM_{t,x}\mcM_{v,loc}}^{\frac{\iota\s}{\iota\s+1+\eta}}\|\chi\|_{L_{t}^{\infty}L_{x,loc}^{1}BV_{v}}^{\frac{1+\eta}{\iota\s+1+\eta}}\\
 & +z^{\frac{1}{1+\s}+\frac{1}{\iota\s+1+\eta}}\|m\|_{\mcM_{t,x}\mcM_{v,loc}}^{\frac{\iota\s+\eta}{\iota\s+1+\eta}}\|\chi\|_{L_{t}^{\infty}L_{x,loc}^{1}BV_{v}}^{\frac{1}{\iota\s+1+\eta}}\\
 &+z^\frac{2+\eta}{\iota\s+1\eta} \|m\|_{\mcM_{t,x}\mcM_{v,loc}}^{\frac{\iota\s-1}{\iota\s+1+\eta}}\|\chi\|_{L_{t}^{\infty}L_{x,loc}^{1}BV_{v}}^\frac{2+\eta}{\iota\s+1+\eta}.
\end{align*}

Since, for $\s\in(0,1)$,  $\frac{1+\eta}{\iota\s+1+\eta}\le\frac{1}{1+\s}+\frac{1}{\iota\s+1+\eta}$ and $|z|\le1$ we obtain
\begin{align}
K(u,z) & \lesssim z^{\frac{1+\eta}{\iota\s+1+\eta}}\|\chi\|_{L_{t}^{\infty}L_{x,loc}^{1}BV_{v}}^{\frac{1+\eta}{\iota\s+1+\eta}}\|m\|_{\mcM_{t,x}\mcM_{v,loc}}^{\frac{\iota\s+\eta}{\iota\s+1+\eta}}.
\end{align}
We conclude, with $\theta=\frac{1+\eta}{\iota\s+1+\eta}$,
\begin{align*}
 & \|u\|_{(L_{t}^{1}(\dot{W}_{x}^{-1-\ve_{1},1}),L_{t}^{1}(\dot{W}_{x,loc}^{\s,1}))_{\theta,\infty}}\\
 & =\sup_{z>0}z^{-\t}K(u,z)\\
 & =\sup_{1>z>0}z^{-\t}K(u,z)+\sup_{z\ge1}z^{-\t}K(u,z)\\
 & \le \|\chi\|_{L_{t}^{\infty}L_{x,loc}^{1}BV_{v}}^{\frac{1+\eta}{\iota\s+1+\eta}}\|m\|_{\mcM_{t,x}\mcM_{v,loc}}^{\frac{\iota\s+\eta}{\iota\s+1+\eta}}+\|u\|_{L_{t}^{1}(\dot{W}_{x}^{-1-\ve_{1},1})}.
\end{align*}
Interpolation gives, choosing $\ve_{1}>0$ small enough,
\begin{align*}
(L_{t}^{1}(\dot{W}_{x}^{-1-\ve_{1},1}),L_{t}^{1}(\dot{W}_{x,loc}^{\s,1}))_{\theta,\infty} & \subseteq L_{t}^{1}(\dot{W}_{x,loc}^{s,r})
\end{align*}
for every
\begin{align*}
s & <(1-\theta)(-1)+\s\theta=\s\left(\frac{1+\eta-\iota}{1+\eta+\iota\sigma}\right)\\
\frac{1}{r} & =\frac{1-\theta}{1}+\theta=1.
\end{align*}
In conclusion, since $\s\in(0,1)$ arbitrary (large enough), for every $s<s_{*}=\frac{1+\eta-\iota}{1+\eta+\iota}$,
\begin{align*}
\|u\|_{L_{t}^{1}(\dot{W}_{x,loc}^{s,1})} & \le \|\chi\|_{L_{t}^{\infty}L_{x,loc}^{1}BV_{v}}^{\frac{1+\eta}{\iota+1+\eta}}\|m\|_{\mcM_{t,x}\mcM_{v,loc}}^{\frac{\iota+\eta}{\iota+1+\eta}}+\|u\|_{L_{t}^{1}(\dot{W}_{x}^{-1,1})}.
\end{align*}

\end{proof}
\bibliographystyle{abbrv}
\bibliography{/home/bgess/Dropbox/bgess/cloud/current_work/latex-refs/refs}

\def\cprime{$'$}
\begin{bibdiv}
\begin{biblist}

\bib{BCD11}{book}{
      author={Bahouri, Hajer},
      author={Chemin, Jean-Yves},
      author={Danchin, Raphae{\"e}l},
       title={Fourier analysis and nonlinear partial differential equations},
      series={Grundlehren der Mathematischen Wissenschaften [Fundamental
  Principles of Mathematical Sciences]},
   publisher={Springer, Heidelberg},
        date={2011},
      volume={343},
        ISBN={978-3-642-16829-1},
}

\bib{BL76}{book}{
      author={Bergh, J{\"o}ran},
      author={L{\"o}fstr{\"o}m, J{\"o}rgen},
       title={Interpolation spaces. {A}n introduction},
   publisher={Springer-Verlag, Berlin-New York},
        date={1976},
        note={Grundlehren der Mathematischen Wissenschaften, No. 223},
}

\bib{CG16}{article}{
      author={Catellier, R.},
      author={Gubinelli, M.},
       title={Averaging along irregular curves and regularisation of {ODE}s},
        date={2016},
     journal={Stochastic Process. Appl.},
      volume={126},
      number={8},
       pages={2323\ndash 2366},
}

\bib{DLW03}{article}{
      author={De~Lellis, Camillo},
      author={Westdickenberg, Michael},
       title={On the optimality of velocity averaging lemmas},
        date={2003},
     journal={Ann. Inst. H. Poincar\'e Anal. Non Lin\'eaire},
      volume={20},
      number={6},
       pages={1075\ndash 1085},
}

\bib{DV13}{article}{
      author={Debussche, A.},
      author={Vovelle, J.},
       title={Invariant measure of scalar first-order conservation laws with
  stochastic forcing},
        date={2015},
     journal={Probab. Theory Related Fields},
      volume={163},
      number={3-4},
       pages={575\ndash 611},
}

\bib{GG16}{article}{
      author={Gassiat, Paul},
      author={Gess, Benjamin},
       title={Regularization by noise for stochastic {H}amilton-{J}acobi
  equations},
        date={2016},
     journal={arXiv:1609.07074},
}

\bib{GS14-2}{article}{
      author={Gess, Benjamin},
      author={Souganidis, Panagiotis~E.},
       title={Long-time behaviour, invariant measures and regularizing effects
  for stochastic scalar conservation laws},
        date={2017},
     journal={to appear in: Comm. Pure Appl. Math.},
}

\bib{GS17}{article}{
      author={Gess, Benjamin},
      author={Souganidis, Panagiotis~E.},
       title={Long-time behaviour, invariant measures and regularizing effects
  for stochastic scalar conservation laws - revised version},
        date={2017},
     journal={preprint},
}

\bib{JP02}{article}{
      author={Jabin, Pierre-Emmanuel},
      author={Perthame, Beno{\^{\i}}t},
       title={Regularity in kinetic formulations via averaging lemmas},
        date={2002},
     journal={ESAIM Control Optim. Calc. Var.},
      volume={8},
       pages={761\ndash 774 (electronic)},
        note={A tribute to J. L. Lions},
}

\bib{JV04}{article}{
      author={Jabin, Pierre-Emmanuel},
      author={Vega, Luis},
       title={A real space method for averaging lemmas},
        date={2004},
     journal={J. Math. Pures Appl. (9)},
      volume={83},
      number={11},
       pages={1309\ndash 1351},
}

\bib{LPS13-2}{incollection}{
      author={Lions, Pierre-Louis},
      author={Perthame, Beno{\^{\i}}t},
      author={Souganidis, Panagiotis~E.},
       title={Stochastic averaging lemmas for kinetic equations},
        date={2013},
   booktitle={S\'eminaire {L}aurent {S}chwartz---\'{E}quations aux d\'eriv\'ees
  partielles et applications. {A}nn\'ee 2011--2012},
      series={S\'emin. \'Equ. D\'eriv. Partielles},
   publisher={\'Ecole Polytech., Palaiseau},
       pages={Exp. No. XXVI, 17},
}

\bib{LPS14}{article}{
      author={Lions, Pierre-Louis},
      author={Perthame, Beno{\^{\i}}t},
      author={Souganidis, Panagiotis~E.},
       title={Scalar conservation laws with rough (stochastic) fluxes: the
  spatially dependent case},
        date={2014},
     journal={Stoch. Partial Differ. Equ. Anal. Comput.},
      volume={2},
      number={4},
       pages={517\ndash 538},
}

\bib{LPS13}{article}{
      author={Lions, Pierre-Louis},
      author={Perthame, Benoît},
      author={Souganidis, Panagiotis~E.},
       title={Scalar conservation laws with rough (stochastic) fluxes},
        date={2013},
     journal={Stoch. Partial Differ. Equ. Anal. Comput.},
      volume={1},
      number={4},
       pages={664\ndash 686},
}

\end{biblist}
\end{bibdiv}

\end{document}